\documentclass[review]{elsarticle}

\usepackage{hyperref} 

\usepackage{amssymb,mathtools,amsthm,xcolor,algorithm,algpseudocode,multicol}
\usepackage{natbib}
\bibpunct[, ]{(}{)}{,}{a}{}{,}%
\def\BIBand{and}%

\theoremstyle{plain}
\newtheorem{theorem}{Theorem}[section]
\newtheorem{lemma}[theorem]{Lemma}

\theoremstyle{definition}
\newtheorem{definition}{Definition}[section]

\theoremstyle{remark}
\newtheorem{remark}{Remark}

\biboptions{authoryear}

\begin{document}

\begin{frontmatter}

\title{A unified algorithm framework for mean-variance optimization in discounted Markov decision processes} 


\author[address1]{Shuai Ma} 

\ead{mash35@mail.sysu.edu.cn}

\author[address2]{Xiaoteng Ma} 

\ead{ma-xt17@mails.tsinghua.edu.cn}

\author[address1]{Li Xia\corref{mycorrespondingauthor}}

\address[address1]{School of Business, Sun Yat-sen University, Guangzhou, 510275, P. R. China}
\address[address2]{Department of Automation, Tsinghua University, Beijing, 100086, P. R. China}



\cortext[mycorrespondingauthor]{Corresponding author. Email: \url{xiali5@sysu.edu.cn}}

%
%
%

\begin{abstract}
This paper studies the risk-averse mean-variance optimization in
infinite-horizon discounted Markov decision processes (MDPs). The
involved variance metric concerns reward variability during the
whole process, and future deviations are discounted to their present
values. This discounted mean-variance optimization yields a reward
function dependent on a discounted mean, and this dependency renders
traditional dynamic programming methods inapplicable since it
suppresses a crucial property---time consistency. To deal with this
unorthodox problem, we introduce a pseudo mean to transform the
untreatable MDP to a standard one with a redefined reward function
in standard form and derive a discounted mean-variance performance
difference formula. With the pseudo mean, we propose a unified
algorithm framework with a bilevel optimization structure for the
discounted mean-variance optimization. The framework unifies a
variety of algorithms for several variance-related problems
including, but not limited to, risk-averse variance and
mean-variance optimizations in discounted and average MDPs.
Furthermore, the convergence analyses missing from the literature
can be complemented with the proposed framework as well. Taking the
value iteration as an example, we develop a discounted mean-variance
value iteration algorithm and prove its convergence to a local
optimum with the aid of a Bellman local-optimality equation.
Finally, we conduct a numerical experiment on portfolio management
to validate the proposed algorithm.
\end{abstract}


\end{frontmatter}


\section{Introduction}

Financial optimizations usually involve trade-offs between profit
and risk, and variance is to risk what mean is to profit. It could
be the reason why the mean-variance optimization theory initiated by
\cite{Markowitz1952} is one of the most prevalent financial
optimization frameworks.
As a cornerstone of the modern portfolio theory, the mean-variance
optimization theory has been extensively applied in a variety of
financial problems, such as portfolio selection~\citep{Best91},
hedging~\citep{Kouvelis18}, pricing~\citep{Kandel89}, etc. It
appeals to both academia and industry not only for its simplicity
but also for being a satisfactory proxy for other types of risk
minimization rules. \cite{Levy03} show that when diversification
between assets is allowed, the mean-variance optimization theory and
the prospect theory~\citep{Tversky92} almost coincide, which
justifies its robustness. One stream of works on the mean-variance
optimization is from the perspective of stochastic
control~\citep{li2000,basak2010,zhang2012}. In this paper, we study
it from the perspective of Markov decision processes (MDPs), where
it is often assumed that the state and action spaces are finite and
the reward is bounded in a discrete-time scenario.

In the framework of MDPs, a variety of works~\citep{sobel1982variance,Tamar2012policy,Xie2018} 
concern the variance of total discounted reward, or \textit{return}, i.e., $ \mathbb{V}\{\sum_{t=1}^{\infty} \alpha^{t-1} r(X_t)\}  $, given $ \alpha $ the discount factor and $ r(X_t) $ the immediate reward at time epoch $ t $. 
Its counterpart in average MDPs is termed the limiting average variance, which is defined by $ \lim_{T \rightarrow \infty} \frac{1}{T} \mathbb{E}\{ (\sum_{t=1}^{T} r(X_t) - \mathbb{E} \{\sum_{t=1}^{T}r(X_t) \})^2 \} $~\citep{Lerma1999}.
These two variance metrics focus on the fluctuation of cumulative reward at the final epoch.
However, it is seemingly preferable to consider risks at every time point in many practical problems.
For example, an autonomous driving system should pay
attention to every detail along the road to ensure safe driving. In
finance, it is easy to manipulate a stock price at a specified time
point, but it is barely possible to do it during the whole process.

Motivated by the above observations, we consider a steady-state
variance metric in the mean-variance optimization. The steady-state
variance is also known as the long-run variance~\citep{Filar89},
which is defined by $ \lim_{T \rightarrow \infty}  \frac{1}{T} \\
\mathbb{E} \{ \sum_{t=1}^{T} (r(X_t) - \eta_a)^2 \} $. It quantifies
the dispersion of immediate rewards from the long-run average $
\eta_a = \lim_{T \rightarrow \infty} \frac{1}{T} \mathbb{E}\{
\sum_{t=1}^{T} r(X_t) \} $ by \textit{averaging} the reward
deviations during the whole process. The difference between the
limiting average variance and the steady-state variance is discussed
by \cite{XIA2016}. Different from the classic definition, we
introduce a discount factor in the steady-state variance, and
optimize the discounted mean-variance objective $ (\eta - \beta
\zeta) $, where $ \eta $ is the normalized discounted mean, $ \beta
$ is a risk-aversion parameter, and $ \zeta = \mathbb{E}\{
\sum_{t=1}^{\infty} \alpha^{t-1} (r(X_t) - \eta)^2 \} $ is the
\emph{discounted steady-state variance}. The motivation of involving
the discount factor in the variance metric is twofold. First, we
calculate the \textit{present} risk (deviation) value from a
monetary point of view, which renders risks at multiple epochs
comparable with a consistent measure. Second, a discounted variance
puts more emphasis (weights $\alpha^{t-1}$) on the transient
behavior at the beginning of the process, which may account for a
property of the risk: future risks are less critical than current one.
Moreover, it is more tractable both
computationally and analytically with a discount factor from a mathematical viewpoint.

For the steady-state variance, \cite{Filar89} illustrate the reasonability of the steady-state variance in average MDPs with a simple example. 
They point out that when variance is concerned, the discounted
problem is more difficult to analyze than its average counterpart.
Furthermore, they model the two problems as convex quadratic
programs and prove the existence of deterministic optimal policies.
\cite{Sobel94} and \cite{Chung94} independently analyze the variance
optimization problem with a mean performance constraint in unichain
MDPs. With the aid of the extant theory on quasiconcave
minimization, the problem can be transformed to a linear program,
and the relevant properties and Pareto optimality are studied.
\cite{Prashanth2013NIPS} propose actor-critic algorithms to estimate
policy gradients of the return variance in discounted MDPs and
the steady-state variance in average MDPs. With the ordinary
differential equation approach, they prove the asymptotic local
convergences of the algorithms. \cite{Gosavi2014} proposes a
model-free algorithm analogous to Q-learning for the mean-variance
problem in average reinforcement learning (RL). The algorithm is
validated with a numerical experiment, but the convergence analysis is missing. This gap is filled
by~\cite{XIA2016}, who proposes a policy iteration for variance
minimization in average MDPs, regardless of the mean performance.
With the aid of the sensitivity-based optimization
theory~\citep{Cao2007}, \citeauthor{XIA2016} derives a variance
performance difference formula (PDF), which quantifies the variance
difference between MDPs under any two policies. With the PDF, the local
convergence of the proposed policy iteration is proved. This work is
later extended to the mean-variance optimization in average
MDPs~\citep{Xia2020}. \cite{Bisi2020ijcai} study the discounted
mean-variance in RL, where the steady-state variance is evaluated to
bound the limiting average variance. They develop a gradient-based
trust region policy optimization (originally proposed by
\cite{schulman15}) algorithm with a monotonic policy improvement.
\cite{Zhang2021} focus on the mean-variance optimization in both
discounted and average MDPs, and develop a policy iteration
algorithm and a gradient-based RL algorithm. To deal with the
policy-dependent reward function, they reformulate variance with its
Legendre-Fenchel dual with an extra variable introduced. Virtually,
this variance reformulation is similar to the ones with pseudo
variances defined by \cite{XIA2016,Xia2020}, where more detailed
analyses are given on variance and mean-variance optimization
problem equivalences and local convergences in average MDPs,
respectively.

Two problems emerge from the literature review.
One is that the mean-variance optimization has been studied in discounted and average MDPs separately, and the relationship between the two cases has not been revealed. 
The other problem is that various optimization algorithms
are proposed without convergence analyses, which is partially
because the variance-related criteria do not fit in with a standard
MDP model. Variance is a quadratic function of mean, and a
variance-related problem is equivalent to an MDP with a special
reward function, whose value for each state depends on policy
instead of action, i.e., the variance value function at a current state
will be affected by actions chosen at not only the current epoch but
also future epochs. This dependency deprives the time-consistency
property and revokes traditional dynamic programming (DP)
methods~\citep{Puterman1994a,Eckstein16,Bisi2020ijcai}. In other words, the
Bellman optimality equation does not optimize over the admissible
action set for a state $ x $ ($ \max_{a \in A(x)} $), but over the
policy space for the whole state space $ S $ ($ \max_{d \in D} $),
and we can not divide and conquer this problem in a traditional
manner. Most of the relevant studies resort to program-based or
gradient-based methods, but the first type of methods cannot deal
with problems with large state and action spaces, and the second usually
suffers from intrinsic deficiencies: slow convergence, large
variance of gradient estimates, and sensitivity to step
sizes~\citep{ZHAO2012}. \cite{XIA2016,Xia2020} proposes policy
iterations for the risk-averse variance and mean-variance
optimizations in average MDPs, which offer a new perspective for
variance-related problems.

In this paper, we first unify the two mean-variance optimization
problems with the continuity property of the discounted
mean-variance metric, and we show that the average problem
formulation can be viewed as a special case when $ \alpha \uparrow 1
$. For details see Remark~\ref{remark1}. This formulation
unification offers a systematic perspective in contrast to the
previous works. To deal with the difficulty caused by policy
dependency, we analyze it in the theory of sensitivity-based
optimization~\citep{Cao2007}, which stems from the perturbation
analysis theory~\citep{Ho2012} and has been largely extended to
stochastic dynamic systems including Markov models. 
This theory is applicable to general Markov systems, even including
the cases without the time-consistency property. With the
sensitivity-based optimization theory, we derive a discounted
mean-variance PDF, based on which we propose a unified algorithm
framework with a bilevel optimization structure, where the inner
problem concerns a standard MDP with a fixed pseudo mean, and the
outer problem refers to a one-dimensional optimization of the pseudo
mean. Different algorithms can be developed with convergence
analyses, which are crucial for the risk-averse variance-related
problems. Taking the value iteration as an example, we propose a
discounted mean-variance value iteration (DMVVI) algorithm and prove
its local convergence. In addition, we present a Bellman
local-optimality equation, which presents a necessary and sufficient
condition for local optimality of a policy. Finally, we apply the
DMVVI algorithm to a portfolio management problem and illustrate its
validity.

The contributions of our paper are twofold.
First, we present a unified algorithm framework for the risk-averse
(mean-)variance optimization problem in discounted and average MDPs.
This unified framework can unify the algorithms in relevant works and
provide a new perspective for the dynamic optimization concerning steady-state variance metrics.
Second, 
we develop a DMVVI algorithm and prove its convergence, which can
provide a foundation for further developing efficient
temporal-difference learning methods, such as Q-learning,
SARSA~\citep{sutton1998reinforcement}, and RL with neural networks, to variance-related optimization problems.
We believe that the algorithm framework and the DMVVI algorithm can
complement the steady-state variance optimization theory together
with the existing works of policy
iterations~\citep{XIA2016,Xia2020}.

The remainder of the paper proceeds as follows. Section~2 formulates
the risk-averse mean-variance optimization in discounted MDPs.
Section~3 
proposes a unified algorithm framework with a bilevel optimization
structure. Several algorithms can be developed in this framework,
and we propose a DMVVI as an example and prove its convergence.
Section~4 gives a numerical experiment on financial dynamic
portfolio management to validate the DMVVI algorithm.
Section~5 presents concluding comments.

\section{Problem formulation}
In this paper, we focus on infinite-horizon discrete-time MDPs, which can be represented by
$ \mathcal{M} = \langle S, A, r, p, \mu, \alpha \rangle $,
in which
$S = \{s_1, s_2, \cdots s_{|S|}\} $ is a finite state space, and $X_t \in S$ represents the state at (decision) epoch $t \in \mathbb{N}^+ = \{1,2,\cdots\}$;
$A(x)$ is the admissible action set for $x \in S$, $A = \bigcup_{x \in S}A(x)$ is a finite action space, and $K_t \in A$ represents the action at $t$;
$r:S \times A \rightarrow \mathbb{R}$ is a bounded reward function;
$p(y \mid x, a) = \mathbb{P}(X_{t+1}=y \mid X_t = x, K_t = a)$ denotes the homogeneous transition probability;
$ \mu: S \rightarrow [0,1] $ is the initial state probability mass function, and $\boldsymbol{\mu}$ is an $ |S| $-dimensional row vector with the $ n $-th entry $ \mu(s_n) $ for any $ n \in \{1, \cdots , |S|\} $;
and $ \alpha \in (0,1)$ is the discount factor.

A policy describes how to choose actions sequentially.
It is stationary when it is independent of time, and deterministic if it determines an action for each state.
In this study, we focus on stationary deterministic policy space $ D $ only.
For a given MDP, a policy $ d: S \rightarrow A $ induces a Markov reward process.
We denote its transition probability by an $ |S| $-by-$ |S| $ matrix $ \mathbf{P}_d $ with $ P_d(x, y) = p(y \mid x, d(x)) $, and its reward function by an $ |S| $-dimensional column vector $ \mathbf{r}_d $ with $ r_d(x) = r(x, d(x)) $, for $ x,y \in S $.
We further denote its stationary distribution by an $ |S| $-dimensional row vector $ \boldsymbol{\pi}_d $, with the $ n $-th entry $ \pi_d(s_n) $ for any $ n \in \{1, \cdots , |S|\} $.
For notational simplicity, we omit the subscript 
``$ d $'' when it is clear in the context. 

In this study, we concern a risk-averse discounted mean-variance objective, where the discounted variance refers to the (normalized) cumulative discounted reward deviations from the discounted mean.
Firstly, we denote the discounted mean value function under a policy $ d \in D $ by 
\begin{equation}
	\label{discMean1}
	v(x) = v_{d}(x) \coloneqq (1 - \alpha) \mathbb{E}^d_x \left \{ \sum_{t=1}^{\infty} \alpha^{t-1} r(X_t)  \right \}, \quad x \in S,
\end{equation}
where $ \mathbb{E}^d_x $ stands for the expectation given the initial state $ X_1 = x $ under the policy $ d $, and we derive the following matrix form
\begin{equation} \nonumber
	\mathbf{v} = (1 - \alpha) (\mathbf{I} - \alpha \mathbf{P})^{-1} \mathbf{r}.
\end{equation}
Considering the initial state distribution $ \mu $, we have the discounted mean as
\begin{equation}
	\label{discMean}
	\eta = \eta_{d} \coloneqq (1 - \alpha) \mathbb{E}^d_\mu \left\{ \sum_{t=1}^{\infty} \alpha^{t-1} r(X_t) \right \}   = \boldsymbol{\mu} \mathbf{v},
\end{equation}
where $ \mathbb{E}^d_\mu $ stands for the expectation given $ X_1 \sim \mu $ under the policy $ d $.

Next, we define the discounted steady-state variance with a second moment value function.
The (normalized) discounted second moment value function under a policy $ d $ is
\begin{equation}
	\label{discVar1}
	w(x) = w_{d}(x) \coloneqq (1 - \alpha) \mathbb{E}^d_x \left \{\sum_{t=1}^{\infty} \alpha^{t-1} r^2(X_t) \right \}, \quad x \in S,
\end{equation}
and in matrix form, we have
\begin{equation} \nonumber
	\label{discVar1InMat}
	\mathbf{w} = (1 - \alpha) (\mathbf{I} - \alpha \mathbf{P})^{-1} (\mathbf{r})^2_\odot,
\end{equation}
where ``$ \odot $'' refers to the Hadamard product, i.e., $ (\mathbf{r})^2_\odot = ( r^2(s_1), \cdots, r^2(s_{|S|}))^T $.
Considering the initial state distribution $ \mu $, we have the (normalized) discounted steady-state variance as
\begin{equation} \nonumber
	\label{discVar}
	\zeta = \zeta_{d} \coloneqq (1 - \alpha) \mathbb{E}^d_\mu \left\{ \sum_{t=1}^{\infty} \alpha^{t-1} [r(X_t) - \eta]^2 \right \}   = \boldsymbol{\mu} (\mathbf{w} - 2 \eta \mathbf{v} + \eta^2 \mathbf{e}),
\end{equation}
where $ \mathbf{e} = (1, \cdots, 1)^T $.


After defining the two value functions,  
we define the discounted mean-variance value function by
\begin{equation}
	\label{mean-variance}
	u(x) = u_{d}(x) \coloneqq v (x) - \beta [w(x) - 2 \eta v(x) + \eta^2 ], \quad x \in S,
\end{equation}
where $ \beta > 0 $ is a risk-aversion parameter.
We may consider the discounted mean-variance optimization problem as a discounted MDP with a special \textit{policy-dependent} reward function represented by
\begin{equation}
	\label{fReward}
	f(x) = f_{d}(x) \coloneqq r(x) - \beta [r(x) - \eta]^2, \quad x \in S,
\end{equation}
where $ \eta $ depends on the policy.
In matrix form, we have
\begin{equation} \nonumber
	\mathbf{f} = \mathbf{r} - \beta (\mathbf{r} - \eta \mathbf{e})^2_\odot,
\end{equation}
and then we have the discounted mean-variance value function in matrix form as
\begin{equation} \nonumber
	\label{mean-varianceInMat}
	\mathbf{u} = \mathbf{v} - \beta (\mathbf{w} - 2 \eta \mathbf{v} + \eta^2 \mathbf{e}) = (1 - \alpha) (\mathbf{I} - \alpha \mathbf{P})^{-1} \mathbf{f}. 
\end{equation}
Considering the initial state distribution $ \mu $, we have the discounted mean-variance as
\begin{equation}
	\label{obj}
	\begin{aligned}
		\xi = \xi_{d}
		& \coloneqq (1 - \alpha) \mathbb{E}^d_\mu \left\{\sum_{t=1}^{\infty} \alpha^{t-1} \{ r(X_t) - \beta [r(X_t) - \eta]^2  \} \right\} \\
		& = (1 - \alpha) \mathbb{E}^d_\mu \left\{\sum_{t=1}^{\infty} \alpha^{t-1} f(X_t)  \right\} = \eta - \beta \zeta = \boldsymbol{\mu} \mathbf{u}.
	\end{aligned}
\end{equation}
Our objective is to find a deterministic policy $ d \in D $ to maximize the discounted mean-variance, i.e.,  
\begin{equation}
	\label{problemDef}
	\begin{aligned}
		\xi^* &\coloneqq \max_{d \in D} \{\xi\}, \\
		d^* &\in \mathop{\arg \max}_{d \in D} \{\xi\}.
	\end{aligned}
\end{equation}
The two equations together define the risk-averse discounted mean-variance optimization problem.
We may consider this problem as a discounted MDP with a reward function defined in \eqref{fReward}, where $ \eta $ is the discounted mean defined in \eqref{discMean}.
The value of the variance part in this special reward function for each state depends on policy instead of action, i.e., the performance at a current state will be affected by actions chosen at not only the current epoch but also future epochs.
This dependency deprives the discounted variance metric of the time-consistency property.
In this case, the Bellman optimality equation does not optimize over the admissible action set for a state, but over the policy space for the whole state space, which revokes the divide-and-conquer DP methods.
In the next section, we will turn to the sensitivity-based optimization theory to solve this problem. 

\begin{remark}[Problem formulation unification by discounting]
	\label{remark1}
	Besides the rationales given in the introduction section, the involvement of a discount factor unifies variance-related problems in both discounted and average MDPs.
	This unification is embodied in the continuities of the discounted mean and variance at $ \alpha = 1 $ with
	\begin{equation} \nonumber
		\label{discContinue}
		\lim_{\alpha \uparrow 1} (1 - \alpha)(\mathbf{I} - \alpha \mathbf{P})^{-1} = \mathbf{e} \boldsymbol{\pi}.
	\end{equation}
	For details, see Chapter 2 in~\citep{Cao2007}.
	In other words, with respect to an expected return/variance/mean-variance objective, the three scenarios are equivalent:
	\begin{enumerate}
		\item an MDP with a discount factor $ \alpha \uparrow 1 $;
		\item a discounted MDP with a special initial state distribution $ \boldsymbol{\mu} = \boldsymbol{\pi} $ (see Lemma~\ref{futility}); and
		\item an average MDP.
	\end{enumerate}
	
\end{remark}
In particular, the discount factor is trivial when the steady-state distribution equals the initial state distribution, for which we have the following lemma.

\begin{lemma}[The futility of discounting] 
	\label{futility}
	For a given policy $ d \in D $, the discounted mean-variance is independent of the discount factor if $ \boldsymbol{\mu} = \boldsymbol{\pi} $, i.e., the initial state distribution equals the stationary distribution.
\end{lemma}

\begin{proof}
	Since $ \boldsymbol{\pi} \mathbf{P} = \boldsymbol{\pi} $, for the discounted mean with the initial state distribution $ \boldsymbol{\mu} = \boldsymbol{\pi} $, we have
	\begin{align*}
		\eta  
		&= \boldsymbol{\pi} (1 - \alpha) (\mathbf{I} - \alpha \mathbf{P})^{-1} \mathbf{r} \\
		&= \boldsymbol{\pi} (1 - \alpha) \left (\sum_{t=1}^{\infty} \alpha^{t-1}\mathbf{P}^{t-1} \right) \mathbf{r} \\
		&= (1 - \alpha) \left (\sum_{t=1}^{\infty} \alpha^{t-1} \boldsymbol{\pi} \mathbf{P}^{t-1} \right) \mathbf{r} \\
		&= (1 - \alpha) \sum_{t=1}^{\infty} \alpha^{t-1}\boldsymbol{\pi} \mathbf{r} \\
		&= \boldsymbol{\pi} \mathbf{r}.
	\end{align*}
	Hence, the discounted mean is independent of the discount factor when $ \boldsymbol{\mu} = \boldsymbol{\pi} $~\citep{sutton1998reinforcement}, and we can similarly derive that $ \zeta = \boldsymbol{\pi} (\mathbf{r} - \eta \mathbf{e})^2_\odot $ in this case. Therefore, we have $ \xi = \boldsymbol{\pi} \mathbf{f} $, i.e., the discounted mean-variance is independent of the discount factor when $ \boldsymbol{\mu} = \boldsymbol{\pi} $. 
\end{proof}

\section{Unified algorithm framework and discounted mean-variance value iteration}

In this section, we propose a unified algorithm framework for the risk-averse mean-variance optimization 
and give a value iteration algorithm as an example.
First, we introduce the pseudo mean to remove the policy dependency of the reward function, and derive the discounted mean-variance PDF, which has a square term to handle the error from the introduction of pseudo mean. 
Next, 
we propose a unified algorithm
framework with a bilevel optimization structure, where the inner problem refers to a standard MDP with a reward function dependent on a fixed pseudo mean, and
the outer problem concerns a one-dimensional optimization of the pseudo mean.
We show that risk-averse (mean-)variance optimization can be solved by algorithm variants in the proposed framework.
Finally, we develop a value iteration in the framework for the discounted mean-variance problem.
Furthermore, we prove its local convergence with a Bellman local-optimality equation, which is a necessary and sufficient condition for local optimality of a policy.

\subsection{Performance difference formula}
One key result of the sensitivity-based optimization theory is the performance difference formula (PDF).
Based on the performance sensitivity analysis, a PDF quantifies the difference between system performances under any two policies.
This theory is valid even for unorthodox Markov systems where the traditional DP methods fail~\citep{Cao2007}.
For the concerned mean-variance optimization, Equation~\eqref{fReward} shows that the reshaped reward function depends on the discounted mean $ \eta $, which is unknown and affected by future actions.
To handle this policy dependency, we firstly replace $ \eta $ with a pseudo mean $ \lambda \in \mathbb{R} $, and define a pseudo reward function for any $ d \in D $ by
\begin{equation}
	\label{pseudoReward}
	f_{\lambda}(x) = f_{\lambda, d}(x) \coloneqq r(x) - \beta [r(x) - \lambda] ^2, \quad x \in S,
\end{equation}
and in matrix form, we have
\begin{equation} \nonumber
	\label{pseudoRewardInMat}
	\mathbf{f}_{\lambda} =  \mathbf{r} - \beta (\mathbf{r} - \lambda \mathbf{e})^2_\odot = ( r(s_1) - \beta [r(s_1) - \lambda]^2, \cdots, r(s_{|S|}) - \beta [r(s_{|S|}) - \lambda]^2)^T.
\end{equation}
The corresponding pseudo discounted mean-variance value function under policy $ d \in D $ is
\begin{equation} \nonumber
	\label{pseudoDiscVar1}
	u_{\lambda}(x) \coloneqq (1 - \alpha) \mathbb{E}^d_x \left \{\sum_{t=1}^{\infty} \alpha^{t-1} f_{\lambda}(X_t)  \right \}, \quad x \in S,
\end{equation}
and in matrix form, we have
\begin{equation}
	\label{pseudoDiscVar1InMat}
	\mathbf{u}_{\lambda} = (1 - \alpha) (\mathbf{I} - \alpha \mathbf{P})^{-1} \mathbf{f}_{\lambda}.
\end{equation}
Considering the initial state distribution $ \mu $, we have the pseudo discounted mean-variance as
\begin{equation}
	\label{pseudoObj}
	\xi_{\lambda} = \xi_{\lambda,d} \coloneqq (1 - \alpha) \mathbb{E}^d_\mu \left\{ \sum_{t=1}^{\infty} \alpha^{t-1} f_{\lambda}(X_t)  \right\} = \boldsymbol{\mu} \mathbf{u}_{\lambda}. 
\end{equation}
Now we have a standard MDP with the pseudo reward function \eqref{pseudoReward}, 
and the difference between the pseudo discounted mean-variance $ \xi_{\lambda} $ and the discounted mean-variance $ \xi $ can be measured.
\begin{lemma}[Deviation of pseudo discounted mean-variance]
	\label{lemma2}
	The pseudo discounted mean-variance and the discounted mean-variance have the following relation
	\begin{equation} \nonumber
		\label{pseudoVarDiff}
		\xi_{\lambda} = \xi - \beta (\eta - \lambda )^2.
	\end{equation}
\end{lemma}
\begin{proof}
	From \eqref{pseudoObj}, we have
	\begin{align*}
		\xi_{\lambda}  
		&= (1 - \alpha) \boldsymbol{\mu} (\mathbf{I} - \alpha \mathbf{P})^{-1} \mathbf{f}_{\lambda} \\
		&= (1 - \alpha) \boldsymbol{\mu} (\mathbf{I} - \alpha \mathbf{P})^{-1} [\mathbf{r} - \beta (\mathbf{r} - \lambda \mathbf{e})^2_\odot] \\
		&= (1 - \alpha) \boldsymbol{\mu} (\mathbf{I} - \alpha \mathbf{P})^{-1} [\mathbf{r} - \beta(\mathbf{r} - \eta \mathbf{e} + \eta \mathbf{e} -\lambda \mathbf{e})^2_\odot] \\
		&= (1 - \alpha) \boldsymbol{\mu} (\mathbf{I} - \alpha \mathbf{P})^{-1} \left\{[ \mathbf{r} - \beta (\mathbf{r} - \eta \mathbf{e})^2_\odot] - \beta [2 \eta \mathbf{r} - 2 \lambda \mathbf{r}  - \eta^2 \mathbf{e} + \lambda^2 \mathbf{e} ]\right\} \\
		&= \xi - \beta (1 - \alpha) \boldsymbol{\mu} (\mathbf{I} - \alpha \mathbf{P})^{-1}  [2 \eta \mathbf{r} - 2 \lambda \mathbf{r}  - \eta^2 \mathbf{e} + \lambda^2 \mathbf{e} ].
	\end{align*}
	With \eqref{discMean} and noticing that $ (1 - \alpha) \boldsymbol{\mu} (\mathbf{I} - \alpha \mathbf{P})^{-1} \mathbf{e} = 1 $, we have
	\begin{align*}
		\xi_{\lambda} &= \xi - \beta  (2 \eta^2 - 2 \lambda \eta  - \eta^2  + \lambda^2)  \\
		&= \xi - \beta (\eta - \lambda )^2.  
	\end{align*}
\end{proof}

\begin{remark}[Means in discounted variance]
	One may be tempted to set the long-run average $ \eta_{a} = \lim_{T \rightarrow \infty} \frac{1}{T} \mathbb{E}\{ \sum_{t = 1}^{T} r(X_t) \} $ as the value from which the deviations are measured.
	Though it has a straightforward physical meaning, it is not the real mean in this discounted setting.
	When a discount factor is involved, it implies that the underpinned occupation measure of state-action pairs is in a discounted form, so the real ``baseline'' is the first central moment---the discounted mean $ \eta $. 
	This claim is supported by Lemma~\ref{lemma2} as well, since the real mean should minimize variance (maximize mean-variance). 
\end{remark}

To construct the discounted mean-variance PDF, we start by quantifying the difference between any two pseudo discounted mean-variance functions under two policies with respect to a pseudo mean $ \lambda \in \mathbb{R} $.
From \eqref{pseudoDiscVar1InMat}, for $ d \in D $ we have
\begin{equation} \nonumber
	\mathbf{u}_{\lambda} = (1 - \alpha) \mathbf{f}_{\lambda} + \alpha \mathbf{P} \mathbf{u}_{\lambda}.
\end{equation}
Denote the other pseudo discounted mean-variance function by $ \mathbf{u}'_{\lambda} $ under $ d' \in D $, with the transition matrix $ \mathbf{P}' $ and the pseudo reward $ \mathbf{f}'_{\lambda} $, and then we have the difference as
\begin{align}
	\mathbf{u}'_{\lambda} - \mathbf{u}_{\lambda}  &= (1 - \alpha) (\mathbf{f}'_{\lambda} - \mathbf{f}_{\lambda}) + \alpha (\mathbf{P}' \mathbf{u}'_{\lambda} - \mathbf{P} \mathbf{u}_{\lambda})  \notag \\
	&= (1 - \alpha) (\mathbf{f}'_{\lambda} - \mathbf{f}_{\lambda}) + \alpha (\mathbf{P}' \mathbf{u}'_{\lambda} - \mathbf{P} \mathbf{u}_{\lambda} + \mathbf{P}' \mathbf{u}_{\lambda} - \mathbf{P}' \mathbf{u}_{\lambda})  \notag  \\
	&= (1 - \alpha) (\mathbf{f}'_{\lambda} - \mathbf{f}_{\lambda}) + \alpha (\mathbf{P}' - \mathbf{P}) \mathbf{u}_{\lambda} + \alpha \mathbf{P}' (\mathbf{u}'_{\lambda} - \mathbf{u}_{\lambda})  \label{tmp11}  \\
	&= (1 - \alpha) (\mathbf{I} - \alpha \mathbf{P}')^{-1} (\mathbf{f}'_{\lambda} - \mathbf{f}_{\lambda}) + \alpha (\mathbf{I} - \alpha \mathbf{P}')^{-1} (\mathbf{P}' - \mathbf{P}) \mathbf{u}_{\lambda} \label{tmp12} \\
	&= (\mathbf{I} - \alpha \mathbf{P}')^{-1} [(1 - \alpha)  (\mathbf{f}'_{\lambda} - \mathbf{f}_{\lambda}) + \alpha  (\mathbf{P}' - \mathbf{P}) \mathbf{u}_{\lambda}] \label{standardVI},
\end{align}
noticing from \eqref{tmp11} to \eqref{tmp12}, we have $ (I - \alpha \mathbf{P}') (\mathbf{u}'_{\lambda} - \mathbf{u}_{\lambda}) = (1 - \alpha) (\mathbf{f}'_{\lambda} - \mathbf{f}_{\lambda}) + \alpha (\mathbf{P}' - \mathbf{P}) \mathbf{u}_{\lambda} $.
Equation~\eqref{standardVI} checks the update rule of the standard value iteration from the perspective of PDF, and it explains why the standard value iteration converges to a global optimum.
For the discounted mean-variance optimization, multiply the initial state distribution $ \mu $ on both sides, and then we have the PDF for the pseudo discounted mean-variance as
\begin{equation} \nonumber
	\xi'_{\lambda} - \xi_{\lambda} = \boldsymbol{\mu} (\mathbf{I} - \alpha \mathbf{P}')^{-1}  \left[ (1 - \alpha) (\mathbf{f}'_{\lambda} - \mathbf{f}_{\lambda}) + \alpha (\mathbf{P}' -  \mathbf{P}) \mathbf{u}_{\lambda} \right]  .
\end{equation}
Furthermore, with Lemma~\ref{lemma2} we have the PDF for the discounted mean-variance as
\begin{equation}
	\xi' - \xi =  \boldsymbol{\mu} (\mathbf{I} - \alpha \mathbf{P}')^{-1}  [ (1 - \alpha) (\mathbf{f}'_{\lambda} - \mathbf{f}_{\lambda}) + \alpha (\mathbf{P}' -  \mathbf{P}) \mathbf{u}_{\lambda} ]  + \beta (\eta' - \lambda )^2 - \beta (\eta - \lambda )^2.
	\label{performance difference formula}
\end{equation}
Equation~\eqref{performance difference formula} quantifies the difference between the mean-variance performances under two policies in an MDP with a pseudo reward.
Based on \eqref{performance difference formula}, it is straightforward to develop a policy iteration for the risk-averse (mean-)variance optimization problems.
For the problems in average MDPs see \citep{XIA2016,Xia2020}.

{
	\color{black}
	The involvement of the pseudo mean brings in the last two terms in \eqref{performance difference formula} and makes the variance-related optimization converge to a local optimum.
	To clarify the local optimality, we present the discounted mean-variance performance derivative formula, which is another fundamental concept in the theory of sensitivity-based optimization. Different from the PDF, the derivative formula captures the behavior when the policy changes in a small local region.
	To see that, we first define a mixed policy space with the concept of mixed policy. For any two policies $d, d^\prime \in D$, we define a mixed policy $d^{\delta, d^\prime}$ for $\delta \in (0, 1)$, which follows $d$ with probability $1 - \delta$ and follows $d^\prime$ in the rest. It is easy to verify that $ \mathbf{P^\delta} = \mathbf{P} + \delta(\mathbf{P}' - \mathbf{P}) $ and $ \mathbf{f}^\delta_\lambda = \mathbf{f}_\lambda + \delta(\mathbf{f}'_\lambda - \mathbf{f}_\lambda) $.
	Substituting them into \eqref{performance difference formula}, we derive the performance difference between $d^{\delta, d^\prime}$ and $d$ as
	\begin{equation} \nonumber
		\xi^{\delta} - \xi = \boldsymbol{\mu}(\mathbf{I} - \alpha \mathbf{P}^\delta)^{-1} \delta [ (1 - \alpha) (\mathbf{f}'_{\lambda} - \mathbf{f}_{\lambda}) + \alpha (\mathbf{P}' -  \mathbf{P}) \mathbf{u}_{\lambda} ]  + \beta (\eta^\delta - \lambda )^2 -\beta (\eta - \lambda )^2.
	\end{equation}
	Letting $\delta \to 0$, we obtain the derivative formula in the mixed policy space,
	\begin{equation}
		\frac{ \mathrm{d} \xi}{  \mathrm{d} \delta} = \boldsymbol{\mu} (\mathbf{I} - \alpha \mathbf{P})^{-1}  [ (1 - \alpha) (\mathbf{f}'_{\lambda} - \mathbf{f}_{\lambda}) + \alpha (\mathbf{P}' -  \mathbf{P}) \mathbf{u}_{\lambda} ]  + 2 \beta (\eta - \lambda)\frac{ \mathrm{d}\eta} { \mathrm{d} \delta},
		\label{performance derivative formula}
	\end{equation}
	where $ \lim_{\delta \to 0} \mathbf{P}^\delta = \mathbf{P}$ and
	$
	\lim_{\delta \to 0} \frac{ \mathrm{d} (\eta^\delta - \lambda )^2}{ \mathrm{d} \delta} = 2 (\eta - \lambda )\frac{ \mathrm{d}\eta} { \mathrm{d} \delta}
	$.
}
Next, we present a unified algorithm framework for the risk-averse discounted mean-variance optimization and develop a value iteration algorithm with a provable local convergence.

\subsection{Unified algorithm framework}
In this subsection, we propose a unified algorithm framework for the risk-averse discounted mean-variance optimization.
This framework has a bilevel optimization structure, where the inner problem refers to a standard MDP with a reward function dependent on a fixed pseudo mean, and
the outer problem concerns a one-dimensional optimization of the pseudo mean.
In particular, for variance-related problems (risk-averse discounted/average variance/mean-variance, etc.),
the outer problem has a closed-form solution.
Different algorithm variants can be developed with different solvers to the inner problem.
Moreover, the proposed framework is applicable to some other variance-related optimality criteria as well.

The difficulty of the risk-averse (mean-)variance optimization lies in that the variance metric is a function of the discounted mean.
This dependency suppresses the time-consistency property, and so that the traditional DP methods are not applicable.
To remove the dependency, we introduce the pseudo mean to transform the special MDP to a standard one, where traditional DP methods can be applied.
Mathematically, the introduction of the pseudo mean results in a bilevel optimization problem, which can be further extended to such problem equivalences.
\begin{lemma}[Problem equivalences with pseudo mean]
	\label{lemma4}
	\begin{align}
		\xi^* = \max_{d \in D} \{\xi\}
		&=  \max_{d \in D} \left \{ \max_{\lambda \in \mathbb{R}} \{ \xi_\lambda \} \right \} \label{pseudomean} \\
		&= \max_{\lambda \in \mathbb{R}} \left \{ \max_{d \in D} \{ \xi_\lambda \} \right \}  \label{2level} \\ 
		&= \max_{\lambda \in \mathbb{R}} \left \{ \max_{d \in D} \{ \boldsymbol{\mu} \mathbf{u}_\lambda \} \right \}  \label{PI} \\
		&= \max_{\lambda \in \mathbb{R}} \left\{ \langle (\max_{d \in D}   \{u_\lambda(x)\} )^T_{x \in S}, \boldsymbol{\mu}^T \rangle \right\} \label{pseudoVarOpt1}.
	\end{align}
\end{lemma}
\begin{proof}
	Lemma~\ref{lemma2} implies that \eqref{pseudomean} holds with $ \lambda = \mathop{\arg \max}_{\lambda \in \mathbb{R}} \{ \xi_\lambda \} = \eta_{d} $.
	Since the outer and inner operators are both maximum, the two are exchangeable and \eqref{2level} holds.
	Equation~\eqref{PI} comes from \eqref{pseudoObj}.
	Noticing that for a given $ \lambda $, the inner optimization refers to a standard MDP, and the optimal mean results from the optimal value function $ u_\lambda $.
\end{proof} 

Equation~\eqref{2level} underpins the bilevel algorithm framework.
By introducing the pseudo mean $ \lambda $, the original problem is
transformed to a bilevel problem, where the inner problem concerns
a standard MDP $ \mathcal{M}_\lambda = \langle S, A, f_\lambda, p,
\mu, \alpha \rangle $, and the outer problem refers to a
one-dimensional optimization of the variable $\lambda$. The framework
is shown in Algorithm~\ref{2stageAlgFr}.

\renewcommand{\algorithmicrequire}{\textbf{Input:}}
\renewcommand{\algorithmicensure}{\textbf{Output:}}
\def\NoNumber#1{{\def\alglinenumber##1{}\State #1}\addtocounter{ALG@line}{-1}}
\begin{algorithm}[H]
	\caption{A unified algorithm framework for discounted mean-variance optimization}
	\label{2stageAlgFr}
	\begin{algorithmic} 
		
		\Require
		The MDP $ \mathcal{M} $; initialize two pseudo means $ \lambda , \lambda' \in \mathbb{R}$ with $\lambda' \neq \lambda $  
		\Ensure
		A local optimal policy $d$ and the discounted mean-variance optimum $\xi_{d}$
		
		\While {$ \lambda \neq \lambda' $}
		\State $ \lambda \leftarrow \lambda' $
		\State Construct a standard MDP $\mathcal{M}_{\lambda}$.
		By using standard/optimistic algorithms (e.g., policy iteration, value iteration, or policy gradient), solve or partly solve $\mathcal{M}_{\lambda}$ to obtain an improved policy $d$ and $\lambda'$  \Comment{Inner optimization}
		
		\EndWhile
		\State Return $d$ 
		and $\xi_{d} = \boldsymbol{\mu} \mathbf{u}_{\lambda} $
	\end{algorithmic}
\end{algorithm}

In the bilevel framework, the inner optimization helps optimize $
\lambda $, which means to keep updating it with a value closer to
the discounted mean of any local optimum. For any fixed $ \lambda $,
the resultant $ \mathcal{M}_{\lambda} $ is a standard MDP, so there
are two threads to optimize $ \lambda $. One is to calculate the
optimal discounted mean with a standard DP algorithm. The
convergence of algorithms stemming from this thread is guaranteed by
the convergence of the involved DP algorithm and Lemma~\ref{lemma2}, 
which claims that the pseudo discounted mean-variance $
\xi_{\lambda} $ equals the discounted mean-variance $ \xi $ when $
\lambda = \eta $. This thread is straightforward but could be
conservative. The other thread is to improve $ \lambda $ with an
intermediate value during its process converging to $ \eta_\lambda
$. The variant of policy iteration implementing in this thread is
well known as the optimistic policy
iteration~\citep{sutton1998reinforcement}, which has been studied
for solving the risk-averse variance and mean-variance optimizations
in average MDPs in~\citep{XIA2016,Xia2020}, respectively. Here we
give simplified descriptions on variants of policy iteration and
value iteration for the inner optimization as examples in
Algorithms~\ref{InnerPI} and \ref{InnerVI}~\footnote{The norm used through the paper could be $ p $-norm for $ p \in \{1, 2, +\infty\} $.}~\footnote{For the value
	iteration variants, the policy $ d $ should be derived at the end of
	Algorithm~\ref{2stageAlgFr}, and here we put it in the descriptions
	for structure unification only.}.
The inputs of these four algorithm
variants are the MDP $ \mathcal{M}$ and the current pseudo mean $ \lambda $, and the outputs
are the updated policy $d$ and pseudo mean
$\lambda$. For some realizations of the algorithm variants,
we need to add the additional initializations to the input of the
framework. It is worth noting that, though a set of DP algorithms,
such as policy gradient, linear programming, and policy iteration,
can be applied in the first thread, the convergences of their
optimistic counterparts need further deliberations.

\renewcommand{\algorithmicrequire}{\textbf{Add. Init.:}} 
\begin{algorithm}[h]
	\caption{Policy iteration variants for inner optimization in Algorithm~\ref{2stageAlgFr}}
	\label{InnerPI}
	\begin{algorithmic}
		
		\State
		\vspace{-20pt}
		\begin{multicols}{2}
			
			\State {\it{Standard version:}}
			\State Initialize $ d, d' \in D $ with $ d \neq d' $  
			
			\While {$ d \neq d' $}
			
			\State $ d \leftarrow d' $
			\State $ \mathbf{u}_{\lambda} \leftarrow (1 - \alpha) (\mathbf{I} - \alpha \mathbf{P}_d)^{-1} \mathbf{f}_{\lambda, d} $
			\State $ d' \in \mathop{\arg \max}_{d \in D}    \left \{ (1-\alpha) \mathbf{f}_{\lambda, d} + \alpha \mathbf{P}_d \mathbf{u}_{\lambda}     \right \} $
			
			\EndWhile
			\State $ \lambda' \leftarrow (1 - \alpha) \boldsymbol{\mu} (\mathbf{I} - \alpha \mathbf{P}_d)^{-1} \mathbf{r}_d $
			
			\columnbreak
			\State {\it{Optimistic version:}}
			\Require
			$ d \in D $
			\State $ \mathbf{u}_{\lambda} \leftarrow (1 - \alpha) (\mathbf{I} - \alpha \mathbf{P}_d)^{-1} \mathbf{f}_{\lambda, d} $
			\State $ d \in \mathop{\arg \max}_{d \in D}    \left \{ (1-\alpha) \mathbf{f}_{\lambda, d} + \alpha \mathbf{P}_d \mathbf{u}_{\lambda}     \right \} $
			\State $ \lambda' \leftarrow (1 - \alpha) \boldsymbol{\mu} (\mathbf{I} - \alpha \mathbf{P}_d)^{-1} \mathbf{r}_d $
		\end{multicols}
		
	\end{algorithmic}
\end{algorithm}

\begin{algorithm}[h]
	\caption{Value iteration variants for inner optimization in Algorithm~\ref{2stageAlgFr}}
	\label{InnerVI}
	\begin{algorithmic}
		
		\State
		\vspace{-20pt}
		\begin{multicols}{2}
			\State {\it{Standard version:}}  
			\Require
			a small constant $ \theta > 0 $
			\State Initialize value functions $\mathbf{v}, \mathbf{u}_{\lambda}, \mathbf{u}'_{\lambda} \in \mathbb{R}^{|S|} $ with $ \| \mathbf{u}_{\lambda} - \mathbf{u}'_{\lambda} \| > \theta $
			\While {$ \| \mathbf{u}_{\lambda} - \mathbf{u}'_{\lambda} \| > \theta $}
			
			\State $ \mathbf{u}_{\lambda} \leftarrow \mathbf{u}'_{\lambda} $
			\State {\small $ d \in \mathop{\arg \max}_{d \in D}    \left \{ (1-\alpha) \mathbf{f}_{\lambda, d} + \alpha \mathbf{P}_d \mathbf{u}_{\lambda}     \right \} $}
			\State $ \mathbf{u}'_{\lambda} \leftarrow  (1-\alpha) \mathbf{f}_{\lambda, d} + \alpha \mathbf{P}_d \mathbf{u}_{\lambda} $
			\State $ \mathbf{v} $ $\leftarrow (1-\alpha) \mathbf{r}_d + \alpha \mathbf{P}_d \mathbf{v}$
			
			\EndWhile
			\State $ \lambda' \leftarrow  \boldsymbol{\mu} \mathbf{v} $

			\columnbreak
			\State {\it{Optimistic version:}}
			\Require
			value functions $\mathbf{v}, \mathbf{u}_{\lambda} \in \mathbb{R}^{|S|} $
			\State $ d \in \mathop{\arg \max}_{d \in D}    \left \{ (1-\alpha) \mathbf{f}_{\lambda, d} + \alpha \mathbf{P}_d \mathbf{u}_{\lambda}     \right \} $
			\State $ \mathbf{u}_{\lambda} \leftarrow  (1-\alpha) \mathbf{f}_{\lambda, d} + \alpha \mathbf{P}_d \mathbf{u}_{\lambda} $
			\State $ \mathbf{v} $ $\leftarrow (1-\alpha) \mathbf{r}_d + \alpha \mathbf{P}_d \mathbf{v}$
			\State $ \lambda' \leftarrow  \boldsymbol{\mu} \mathbf{v} $
			
		\end{multicols}

	\end{algorithmic}
\end{algorithm}

A variety of algorithms for variance-related optimization in previous works can be unified and analyzed in the proposed framework.
When the equivalence in \eqref{PI} is concerned, we have 
the policy gradient for the mean-variance optimizations in discounted MDPs~\citep{Bisi2020ijcai}, and the policy iteration for the mean-variance optimization in discounted and average MDPs~\citep{Zhang2021}.
However, no convergence analysis is given in either of the works, such as the analyses for the policy iterations in \citep{XIA2016,Xia2020}. 
Since most, if not all, of the algorithms exploit the variance property described in Lemma~\ref{lemma2} and result in local optima, a convergence analysis is crucial.
In the proposed framework, a convergence analysis can be developed with the aid of a PDF.

\begin{remark}[Unified algorithm framework] 
	The unification in the algorithm framework is twofold.
	\begin{enumerate}
		\item A set of problems potentially solvable by algorithms developed in the framework.
		These problems include, but not limited to, the risk-averse variance and mean-variance optimizations in discounted and average MDPs, and these four metrics can be covered by \eqref{obj}.
		When the discount factor $ \alpha \uparrow 1 $, the problem turns into the mean-variance maximization in average MDPs~\citep{Xia2020,Gosavi2014} (see Remark~\ref{remark1}).
		When the risk-aversion parameter $ \beta $ is large enough with respect to the mean, the problem degrades to the variance minimization problem (for average MDPs, see \citep{XIA2016}).
		\item For the unified set of problems, a set of algorithms can be developed and analyzed in the framework, such as the policy gradients~\citep{Prashanth2013NIPS,Bisi2020ijcai}, the policy iterations~\citep{XIA2016,Xia2020,Zhang2021}, and the value iteration~\citep{Gosavi2014}.
		The missing convergence analyses in some previous works can be developed as well.
		Moreover, both standard and optimistic versions of the DP algorithms can be studied with deliberations on their convergences.
	\end{enumerate}
	
\end{remark}

{ \color{black}
	\begin{remark}[Convergence rate and complexity] 
		In the bilevel optimization framework, the inner problem is a standard MDP for a given pseudo mean $ \lambda $.
		The convergence rate relies on the solver to the inner problem.
		For example, the convergence of the value iteration is linear at rate $ \beta $.
		However, since the mean value function $ v $ is different from the mean-variance value function $ u $, the convergence rate of $ \lambda $ cannot be analyzed similarly.
		The complexity of an algorithm in the bilevel optimization framework depends as well.
		Taking value iteration for example, the complexity for each iteration is $ \mathcal{O}(|S|^2|A|) $.
		A lower bound for the number of iterations needed can be estimated with an error bound at the $ n $-th iteration.
		For any $ d \in \mathop{\arg \max}_{d \in D} \{ \mathbf{f_{\lambda, d}} + \beta \mathbf{P_d} \mathbf{u_{\lambda, n}} \}$, where $ \mathbf{u_{\lambda, n}} $ is the pseudo mean-variance value function at $ n $-th iteration, we have 
		\begin{equation*}
			\| \mathbf{u^d_{\lambda}} - \mathbf{u^*_\lambda} \| \leq \frac{2 \beta^{n-1}}{1 - \beta} \| \max_{d \in D} \{ \mathbf{f_{\lambda, d}} + \beta \mathbf{P_d} \mathbf{u_{\lambda, 1}} \} - \mathbf{u_{\lambda, 1}} \|, 
		\end{equation*}
		where $ \mathbf{u^d_{\lambda}} $ is the pseudo mean-variance value function under $ d $, and $ \mathbf{u_{\lambda, 1}} $ is the initial pseudo mean-variance value function~\citep{Puterman1994a}.
		To seek $ \epsilon $-optimal policies, we have 
		\begin{align*}
			&\frac{2 \beta^{n-1}}{1 - \beta} \| \max_{d \in D} \{ \mathbf{f_{\lambda, d}} + \beta \mathbf{P_d} \mathbf{u_{\lambda, 1}} \} - \mathbf{u_{\lambda, 1}} \| \leq \epsilon, \\
			\Leftrightarrow & n \geq \log_{\beta} \left \{ 
			\frac{\epsilon (1 - \beta) }{2 \| \max_{d \in D} \{ \mathbf{f_{\lambda, d}} + \beta \mathbf{P_d} \mathbf{u_{\lambda, 1}} \} - \mathbf{u_{\lambda, 1}} \|} 
			\right \} + 1.
		\end{align*}
		
	\end{remark}
}

\subsection{Discounted mean-variance value iteration}
In this subsection, we develop a discounted mean-variance value iteration (DMVVI) in the proposed framework as an example.
We show the relationship between the original problem and the one with a pseudo mean.
We prove the local convergence of the DMVVI with a Bellman local-optimality equation, which is a necessary and sufficient condition for local optimality of a policy.
We believe that the DMVVI algorithm provides a foundation for model-free RL methods, such as Q-learning and SARSA, to the variance-related optimizations.

Equation~\eqref{pseudoVarOpt1}  
is the foundation for a value iteration to the risk-averse variance-related optimization.
By further extending \eqref{pseudoVarOpt1}, we derive an optimality equation as  
\begin{equation}
	\label{pseudoVarOpt2}
	\xi^* =   \max_{\lambda \in \mathbb{R}} \Big\{ \langle (\max_{a \in A(x)}    \{  (1 - \alpha) f_\lambda (x,a) + \alpha \sum_{y \in S} p(y \mid x,a)u^*_\lambda(y)      \} )^T_{x \in S} , \boldsymbol{\mu}^T \rangle \Big\}.
\end{equation}
Equation~\eqref{pseudoVarOpt2} forms a Bellman optimality equation parameterized by $ \lambda $ for the inner standard MDP with a fixed $ \lambda $.
To improve the discounted mean-variance value function and evaluate the mean simultaneously, we maintain two value functions in the iteration: the discounted mean value function $ v $ and the second moment value function $ w $.
Given a policy $ d \in D $, we define the value function updates with two Bellman operators:
\begin{equation} \nonumber
	\mathbf{v}' = \mathcal{T}_{v,d} \mathbf{v} \coloneqq (1-\alpha) \mathbf{r} + \alpha \mathbf{P} \mathbf{v}
\end{equation}
and
\begin{equation} \nonumber
	\mathbf{w}' = \mathcal{T}_{w,d} \mathbf{w} \coloneqq (1-\alpha) (\mathbf{r})^2_\odot + \alpha \mathbf{P} \mathbf{w}.
\end{equation}
Now we give the value iteration for the risk-averse discounted mean-variance optimization in Algorithm~\ref{DMVVI}, which is a detailed description of the standard value iteration in Algorithm~\ref{InnerVI} in the framework.

\newif\ifboldnumber
\newcommand{\boldnext}{\global\boldnumbertrue}

\algrenewcommand\alglinenumber[1]{%
	\footnotesize\ifboldnumber\bfseries\fi\global\boldnumberfalse#1:}

\renewcommand{\algorithmicrequire}{\textbf{Input:}}
\begin{algorithm}[!h]
	\caption{The discounted mean-variance value iteration (DMVVI)}
	\label{DMVVI}
	\begin{algorithmic}[1]
		
		\Require
		The MDP $ \mathcal{M} $;
		a small threshold $ \theta > 0$;
		initialize $ \lambda, \lambda' \in \mathbb{R} $ with $ \lambda \neq \lambda' $, two discounted mean value functions, two second moment value functions and a pseudo mean-variance value function $ \mathbf{v}, \mathbf{v}', \mathbf{w}, \mathbf{w}', \mathbf{u}_{\lambda} \in \mathbb{R}^{|S|} $, with $   \| \mathbf{u}_{\lambda} - [\mathbf{v}' - \beta (\mathbf{w}'  - 2 \lambda \mathbf{v}' + \lambda^{2} \mathbf{e} )] \| > \theta $
		\Ensure
		Local optimal policy $ d $ and the local optimum $ \xi $
		\boldnext
		\While {$ \lambda' \neq \lambda $}
		\boldnext
		\State $ \lambda \leftarrow \lambda' $
		\boldnext
		\While {$   \| \mathbf{u}_{\lambda} - [\mathbf{v}' - \beta (\mathbf{w}'  - 2 \lambda \mathbf{v}' + \lambda^{2} \mathbf{e} )] \| > \theta $  \bf{ or } $ \left \| \mathbf{v}' - \mathbf{v} \right \| > \theta $}
		\State $ \mathbf{v} \leftarrow  \mathbf{v}' $
		\State $ \mathbf{w} \leftarrow  \mathbf{w}' $
		\State $ \mathbf{u}_{\lambda} \leftarrow \mathbf{v} - \beta (\mathbf{w}  - 2 \lambda \mathbf{v} + \lambda^2 \mathbf{e} ) $
		\For{$ x \in S $}
		\State
		\begin{equation} \nonumber
			d(x) \in \mathop{\arg \max}_{a \in A(x)}    \left \{ (1-\alpha) f_{\lambda}(x,a) + \alpha \sum_{y \in S} p(y \mid x,a)u_{\lambda}(y)     \right \}
		\end{equation}
		\EndFor
		
		\State $ \mathbf{v}' \leftarrow \mathcal{T}_{v,d} \mathbf{v} $
		\State $ \mathbf{w}' \leftarrow \mathcal{T}_{w,d} \mathbf{w} $
		
		\EndWhile
		
		\State $ \lambda' \leftarrow \boldsymbol{\mu} \mathbf{v}' $
		\EndWhile
		
		\State $ \xi = \boldsymbol{\mu} \mathbf{u}_{\lambda} $
		
	\end{algorithmic}
\end{algorithm}

Although the problem with a pseudo mean is different from the original problem, we have the following theorem to relate these two problems in an iterative algorithm.
\begin{theorem}[Relationship between the two problems]
	\label{2problems}
	For a fixed pseudo mean-variance value function $ u_{\lambda} $, compute the pseudo mean-variance $ \xi_{\lambda} $, the policy $ d \in D $ and $ \eta = \eta_{d} $, and then set $ \lambda = \eta $.
	In the next iteration, if we have $ \xi'_{\lambda} \geq \xi_{\lambda} $, then we have $ \xi' \geq \xi $.
	We have $ \xi' > \xi $ if the first inequality strictly holds.
\end{theorem}
\begin{proof}
	With Lemma~\ref{lemma2}, we have
	\begin{equation} \nonumber
		\label{tmp4}
		\xi' - \xi = [\xi'_{\lambda} + \beta (\eta' - \lambda )^2] - [\xi_{\lambda} + \beta (\eta - \lambda )^2].
	\end{equation}
	By setting $ \lambda = \eta $, we have
	\begin{equation} \nonumber
		\label{tmp5}
		\xi' - \xi = \xi'_{\lambda}  - \xi_{\lambda} + \beta (\eta - \eta' )^2.
	\end{equation}
	Therefore, if $ \xi'_{\lambda} \geq (>) \xi_{\lambda} $, we have $ \xi' \geq (>) \xi $.
\end{proof}

Though the error term in Lemma~\ref{lemma2} can gracefully handle the policy dependency of the variance metric, it takes a toll as well.
Next, we show the condition that the global optimum cannot be reached in one iteration of the DMVVI.
\begin{lemma}[Unreachability of global optimum in one iteration]
	\label{Unreachability}
	In one outer iteration of the DMVVI, ignoring the estimation error, given the current discounted mean-variance $ \xi $ with the discounted mean $ \eta $, 
	and the global optimum $ \xi^* $ with
	\begin{equation}
		\label{tmp6}
		\eta^* = \eta_{ d^* }, \qquad d^* \in \mathop{\arg \min}_{d \in D^*} \{ |\eta_{ d } - \eta | \},
	\end{equation}
	where $  D^* $ is the set of global optimal policies, if
	\begin{equation}
		\label{eq:Unreachability}
		(\xi^* - \xi) \leq \beta (  \eta^* - \eta)^2,
	\end{equation}
	the global optimum cannot be reached in the current iteration.
\end{lemma}
\begin{proof}
	For the current discounted mean-variance function $ u $, we have the transition probability matrix $ \mathbf{P} $ and the discounted mean-variance reward function $ f $.
	Denote $ \xi_{\eta}^* $ the pseudo discounted mean-variance under a global optimal policy $ d^* \in D^* $, with $ \mathbf{P}^* $ and $ f^*_{\eta} $.
	From Lemma~\ref{lemma2}, we have
	\begin{align*}
		(\xi^* - \xi) - \beta (  \eta^* - \eta)^2 
		&= \xi_{\eta}^* - \xi \notag \\
		&= \boldsymbol{\mu} (\mathbf{I} - \alpha \mathbf{P}^*)^{-1}  [(1 - \alpha) (\mathbf{f}^*_{\eta} - \mathbf{f}) + \alpha (\mathbf{P}^* - \mathbf{P}) \mathbf{u}]  \leq 0.
	\end{align*}
	Since each entry of $ \boldsymbol{\mu} (\mathbf{I} - \alpha \mathbf{P}^*)^{-1} $ is nonnegative, there exists at least one $ x \in S $, such that
	\begin{equation*} 
    \begin{aligned}
    	& [(1 - \alpha) \mathbf{f}^*_{\eta} + \alpha \mathbf{P}^*\mathbf{u}](x) 
    	\leq [(1 - \alpha) \mathbf{f} + \alpha \mathbf{P} \mathbf{u}](x) \\
    	& \quad \leq \max_{a \in A(x)}  \left \{  (1 - \alpha) [r(x, a) - \beta (r(x, a) - \eta)^2] + \alpha \sum_{y \in S} p(y \mid x,a)u(y) \right \}, 
    \end{aligned}		 
	\end{equation*}
	which means that, in this iteration, the value function will not converge to the optimal value function under any $ d^* \in D^* $, so the global optimum cannot be reached in the current iteration. 
\end{proof}

Lemma~\ref{Unreachability} claims that, 
given $ \xi $, $ \eta $, and the optimal policy $ d^* $ whose discounted mean performance $ \eta^* $ is closest to $ \eta $ (see \eqref{tmp6}) in the current iteration, if the difference between $ \eta $ and $ \eta^* $ is relatively large (see \eqref{eq:Unreachability}), then
the pseudo mean-variance under $ d^* $ will deteriorate to $ \xi_{\eta}^* \leq \xi $ because of $ \eta $, and the optimum cannot be reached in this iteration.
With the aid of Lemma~\ref{Unreachability}, we can prove the local convergence of the DMVVI.
{
	\color{black}
	We first give the definition of local optimality for the discounted mean-variance optimization, along with the Bellman local-optimality equation.
	\begin{definition}[Local optimality]
		For a policy $d \in D$, if there exists $ \Delta \in (0, 1)$, we always have $\xi_{d} \geq \xi_{d^{\delta, d^\prime}}$ for any $\delta \in (0, \Delta)$, then we say $d$ is a local optimum in the mixed policy space.
	\end{definition}
}
\begin{definition}[Bellman local-optimality equation]
	A policy $d^\# \in D$ is a local optimal policy if and only if its value function $ u^\# $ satisfies the Bellman local-optimality equation
	\begin{equation}
		u^\#(x) = \max_{a \in A(x)}  \left\{  (1 - \alpha) [r(x, a) - \beta (r(x, a) - \eta_{d^\#})^2]
		+ \alpha \sum_{y \in S} p(y  \mid  x,a)u^\#(y) \right\}.
		\label{localOpt}
	\end{equation}
	
\end{definition}

The local convergence of DMVVI is established with the performance derivative formula \eqref{performance derivative formula}, which shows that the converged value function has a nonpositive gradient in any feasible directions.
Next, we give the local convergence proof of the DMVVI.

\begin{theorem}[Local convergence of DMVVI]
	\label{Local convergence}
	The DMVVI converges to a local optimum.
\end{theorem}
\begin{proof}
	First, we prove the convergence of the DMVVI. 
	At $ t $-th step of the outer iteration, we have a pseudo mean $ \lambda_t \in \mathbb{R} $ and a discounted mean-variance $ \boldsymbol{\mu} \mathbf{u}_{\lambda_t, t} $ (i.e., $ \xi_{\lambda_{t}, t} $) at the beginning of the inner optimization (at Line~3 in Algorithm~\ref{DMVVI}).
	After one inner optimization, we have $ \boldsymbol{\mu} \mathbf{u}_{\lambda_t, t+1} \geq \boldsymbol{\mu} \mathbf{u}_{\lambda_t, t} $ (i.e., $ \xi_{\lambda_{t}, t+1} \geq \xi_{\lambda_{t}, t} $, at Line~1), which is guaranteed by the convergence of the standard value iteration~\citep{Puterman1994a}.
	Derive $ \lambda_{t+1}  $ (at Line~2) and then we have $ \boldsymbol{\mu} \mathbf{u}_{\lambda_{t+1}, t+1} \geq \boldsymbol{\mu} \mathbf{u}_{\lambda_t, t+1} $ (i.e., $ \xi_{\lambda_{t+1}, t+1} \geq \xi_{\lambda_t, t+1} $) based on Lemma~\ref{lemma2}, which strictly holds if $ \lambda_{t+1} \neq \lambda_{t} $.
	Since the policy space $ D $ is finite and $ \xi_{\lambda_{t+1}, t+1} \geq \xi_{\lambda_t, t+1} \geq \xi_{\lambda_t, t}$, this algorithm will stop after a finite number of iterations. Thus, the convergence of the DMVVI is proved.
	
	{
		\color{black}
		Second, we prove that the DMVVI converges to a local optimum.
		From Lemma~\ref{Unreachability}, the DMVVI converges to the current discounted mean-variance $ \xi^\# $ if it is local optimum, i.e., it satisfies the Bellman local-optimality equation \eqref{localOpt}.
		Plugging $\lambda = \eta_{d^\#}$ into \eqref{performance derivative formula}, we derive that
		\begin{equation} \nonumber
			\frac{ \mathrm{d} \xi_{d^\#}}{  \mathrm{d} \delta} = \boldsymbol{\mu} (\mathbf{I} - \alpha \mathbf{P^\#})^{-1}  [ (1 - \alpha) (\mathbf{f}^{\# '} - \mathbf{f}^\#) + \alpha (\mathbf{P}' - \mathbf{P}) \mathbf{u}^\# ]. 
		\end{equation}
		Noticing that the elements of $\boldsymbol{\mu} (\mathbf{I} - \alpha \mathbf{P^\#})^{-1}$ are always nonnegative, we conclude that $\frac{ \mathrm{d} \xi_{d^\#}}{  \mathrm{d} \delta} \leq 0$ along any feasible changing direction, indicating that $d^\#$ is a local optimum in the mixed policy space.  
	}	
\end{proof}

Comparing with the standard value iteration, whose global convergence is guaranteed by \eqref{standardVI}, Theorem~\ref{Local convergence} explains why the DMVVI converges a local optimum---within one iteration, the value function update always depends on the former discounted mean, and the error term in Lemma~\ref{lemma2} suppresses the global policy in that iteration (Lemma~\ref{Unreachability}).
Since the optimization procedure is deterministic, the sequence of $ \lambda $'s depends on the initial $ \lambda $ only.
If a pseudo mean derived from a local optimal policy is reached before one from a global optimal policy, then the DMVVI will converge to this local optimum, and in this case, the iteration is ``trapped'' by the Bellman local-optimality equation \eqref{localOpt}.
This local convergence analysis can be generalized to other algorithm variants governed by the unified algorithm framework, such as the works in~\citep{XIA2016,Xia2020,Zhang2021}.
Besides, if all policies share the same discounted mean, the error term vanishes, and the DMVVI will converge to the global optimum.

By introducing a pseudo mean, the PDF quantifies the performance difference between any two policies and provides a foundation for iterative algorithms.
As shown in the introduction section, most of the relevant works focus on gradient-based methods, 
and to the best of our knowledge, only two works concern iterative algorithms for variance-related problems besides \citep{XIA2016,Xia2020}.
One is \citep{Zhang2021}, which reformulates the mean-variance formulation in average MDPs with its Legendre-Fenchel dual, and derives a similar problem formulation as the one with a pseudo mean.
The authors propose a stochastic block coordinate ascent algorithm \citep{Cui18}, which can be unified as a policy iteration in our unified algorithm framework.
No local convergence analysis is given in this work.
The other is \citep{Gosavi2014}, where a value iteration is proposed for the mean-variance optimization in average MDPs.
However, relevant algorithm analyses, such as convergence and local optimality, are circumvented with assumptions.
We believe that our work complements risk-sensitive optimization in MDPs from three aspects:
\begin{enumerate}
	\item The mean-variance optimization theory is extended to discounted MDPs;
	\item A unified algorithm framework is proposed, where a variety of algorithm variants can be unified and analyzed with the aid of PDF, and the framework works for a collection of risk-sensitive criteria including, but not limit to, several variance-related risk measures; and 
	\item The DMVVI is proposed with a convergence analysis and a Bellman local-optimality equation, and it provides a foundation for model-free RL methods, such as Q-learning and SARSA, to the variance-related optimizations.
\end{enumerate}

\section{Numerical experiment}

%
%
%
%
%
%
%
%
%

In this section, we validate the proposed DMVVI by solving the discounted mean-variance optimization in a portfolio management problem~\citep{Tamar2012policy}.
We assume the dynamics of portfolio management to be a stationary stochastic process and model it as an MDP with an appropriate discretization of all relevant continuous variables.

A portfolio is usually composed of two types of assets.
One is the liquid assets (e.g., short-term T-bills), each of which has a fixed interest rate $ r_l $ and can be sold at any epoch $ t \in \mathbb{N}^+$.
The other type is the non-liquid assets (e.g., low liquidity bonds or options), each of which can be sold only after a maturity period of $ M \in \mathbb{N}^+ $ steps with a time-dependent interest rate $ r_n(t) $. 
We assume that $ r_n(t) $ can take either $ r^{low}_n $ or $ r^{high}_n $, and the transitions between these two cases occur randomly with a switching probability $ p_s $.
In addition, a non-liquid asset has a default risk with probability $ p_r $. 
To simplify the problem, we assume that the portfolio has one for each type of assets.
Besides, we discretize the investor's total available cash into $ N \in \mathbb{N}^+ $ units, and represent a state by a vector $ x = (x_0, x_1, \cdots, x_{M+1}) \in \{ 0, \cdots, N \}^{M+1} \times \{ r^{low}_n, r^{high}_n \}$, where
$ x_0 \in \{ 0, \cdots, N \} $ is the number of units invested in the liquid asset;
$ x_1, \cdots, x_{M} \in \{ 0, \cdots, N \} $ are the numbers of units invested in the non-liquid assets with $ 0, \cdots, (M-1) $ 
time steps to maturity, respectively; and
$ x_{M+1} \in \{ r^{low}_n, r^{high}_n \}$ records the current non-liquid interest rate.
At each epoch with a state $ x $, the investor may change her portfolio by investing a number of units $ a \in A(x) = \{0, 1, \cdots, (x_0+x_1)\} $ in the non-liquid asset. 
We further assume that default can happen only at the maturity epoch.
The dynamics of the investment among the liquid asset and the non-liquid assets with different maturity times is illustrated in Figure~\ref{fig:portMan}.
\begin{figure}[!h] 
	\centering
	{\includegraphics[width=0.6\columnwidth]{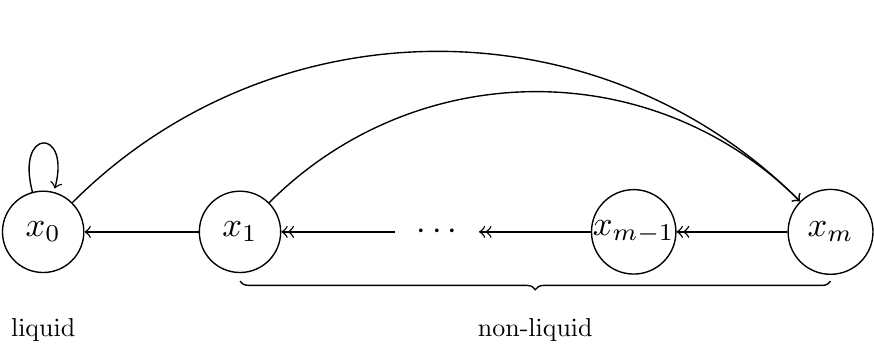}}
	\caption{Dynamics of investment among liquid asset and non-liquid assets with different maturity times. The arrow ``$ \rightarrow $'' represents a controllable investment, and ``$ \twoheadrightarrow $'' represents a uncontrollable state transition.
		Notice that the investment at $ x_1 $ can be directly reinvested to the non-liquid asset since it is matured at the decision epoch.}\label{fig:portMan}
\end{figure}

To consider a small-scale problem, we set
the discount factor $ \alpha = 0.95 $,
the risk-aversion parameter $ \beta  = 1 $,
the maturity period $ M = 3 $,
the total available cash units $ N = 3 $,
the liquid asset interest rate $ r_l = 0.03 $,
the low non-liquid asset interest rate $ r^{low}_n = 0.4 $,
the high non-liquid asset interest rate $ r^{high}_n = 1 $,
the interest switching probability $ p_s = 0.1 $, and
the default risk probability $ p_r =0.1 $.
We assume that all units of cash are in the liquid asset at $ t=1 $.
For the given parameter setting, we construct an MDP to represent this portfolio management problem.



\subsection{Local convergence and laddered policy}
We solve the risk-averse discounted mean-variance optimization in the portfolio management with the DMVVI (Algorithm~\ref{DMVVI}) with $ \theta = 10^{-5} $.
As Theorem~\ref{Local convergence} states, the DMVVI converges to a local optimum.
For different initial pseudo means, the value iteration may converge to different local optima.
In this portfolio management problem with the specified setting, the algorithm converges to the global optimum if we initialize the pseudo mean by $ \lambda' = 1 $.
In contrast, it converges to a local optimum if we initialize $ \lambda' = -1$.
The two convergences are compared in Figures \ref{fig:MVConv} and \ref{fig:ObjConv}.
\begin{figure}[!h] 
	\centering
	{\includegraphics[width=0.8\columnwidth]{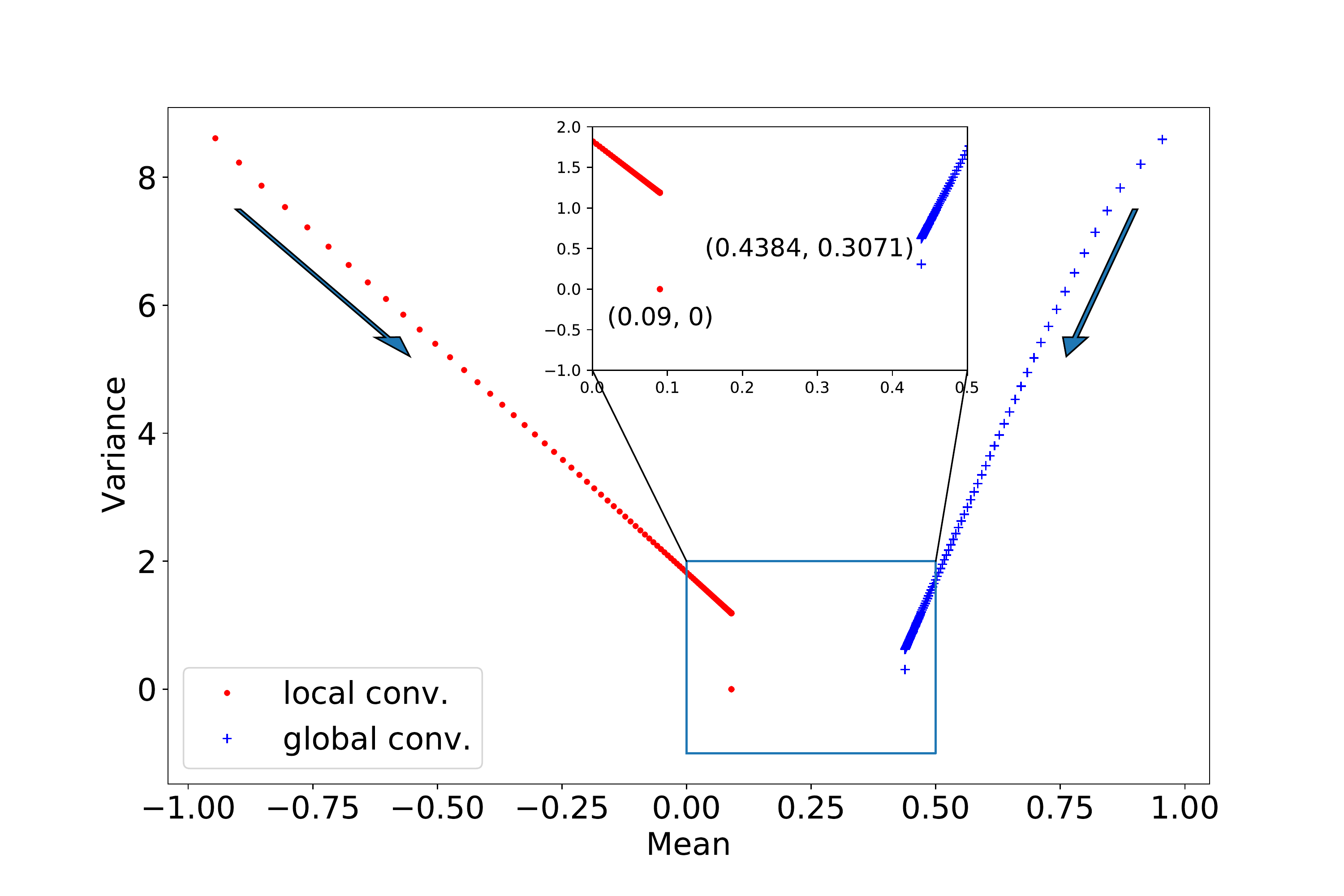}}
	\caption{The local and global convergences in the mean-variance space  with $ \lambda' = -1 $ and $ 1 $, respectively.} \label{fig:MVConv}
\end{figure}
\begin{figure}[!h] 
	\centering
	{\includegraphics[width=0.8\columnwidth]{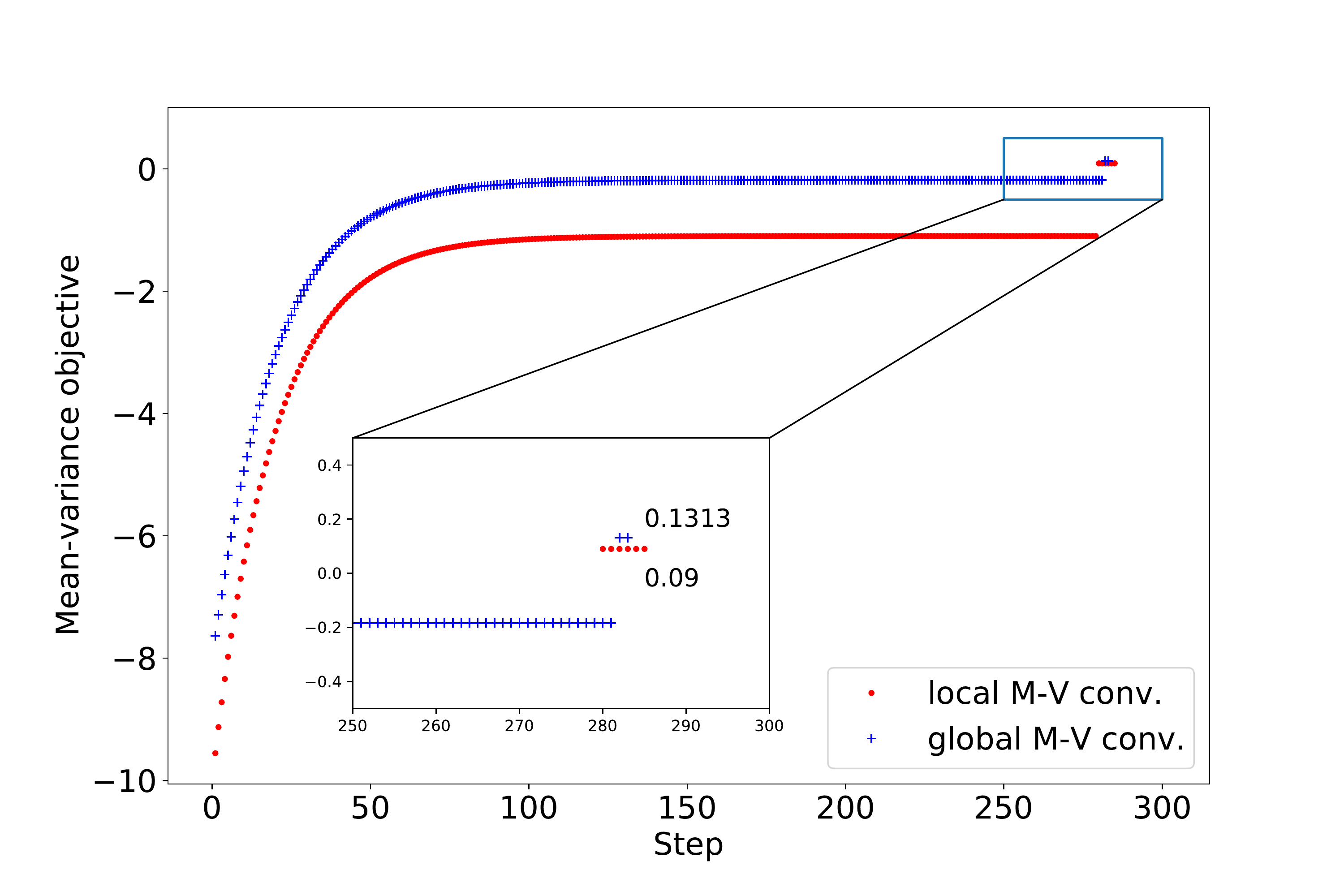}}
	\caption{The local and global convergences along the time steps with $ \lambda' = -1 $ and $ 1 $, respectively.} \label{fig:ObjConv}
\end{figure}

In Figure~\ref{fig:MVConv} we show that when the pseudo mean is initialized differently, the algorithm will converge to different local optima.
In this case, the local optimal policy is $ d(x)=0, x \in S $, which indicates that we should always invest in the liquid risk-free asset, which will deliver a deterministic revenue $ 0.09 $.
Comparing with this conservative local optimal policy, the global optimum achieves a revenue with the discounted mean $ \mu \approx 0.4384 $ and the discounted variance $ \zeta \approx 0.3071 $, which is better with $ \beta  = 1 $.
The discounted mean-variance values are $ 0.1313 $ and $ 0.09 $ for the global and local optima, and the convergences are illustrated in Figure~\ref{fig:ObjConv}.
In both Figures~\ref{fig:MVConv} and \ref{fig:ObjConv}, we can see that there are ``jumps'' near the ends.
That is because the pseudo means are updated in the outer optimization in the bilevel framework.

It is worth noting that, the global optimal policy is $ d(x)=0 $ for $ x \in \{ x \in S  \mid  x_0 + x_1 = 0  \} $ and $ d(x) = 1 $ otherwise. 
This policy is a laddered (or laddering) strategy, which splits an investment to non-liquid assets into equal units and invests them in regular intervals consecutively in order to maintain a cash flow.
Since investments are spread across several maturities, the laddered strategy also implies temporal diversification which reduces financial risks.
Laddered strategies are widely used in portfolio management~\citep{CALDEIRA2016128}.

\subsection{Mean-variance trade-off along convex efficient frontier}
A mean-variance metric is a combination of two metrics weighted with a risk-aversion parameter $ \beta $.
Different $ \beta $'s reflect different trade-offs between profit ($ \mu $) and risk ($ \zeta $).
By adjusting $ \beta $, different policies represented by ($ \mu, \zeta $)'s can be found with the DMVVI (assuming the global optimality is achieved), and these combinations establish a \textit{convex efficient frontier}, which is the intersection of the mean-variance (Pareto) efficient frontier and the convex hull of the mean-variance pairs.
The convex hull, which is also known as the convex envelope or convex closure, 
is the smallest convex set that ``contains'' a given space.
Each policy in the convex efficient frontier corresponds to a global optimal solution to a mean-variance optimization with a specific $ \beta $. 
The convex efficient frontier of the discounted mean-variance problem along with some possible mean-variance pairs is shown in Figure~\ref{fig:PF}. 
In this case, the convex efficient frontier is composed of five vertices, and every vertex represents a possible $ (\mu, \zeta) $ under an optimal policy with respect to some risk-aversion parameter.
\begin{figure}[!h] 
	\centering
	{\includegraphics[width=0.7\columnwidth]{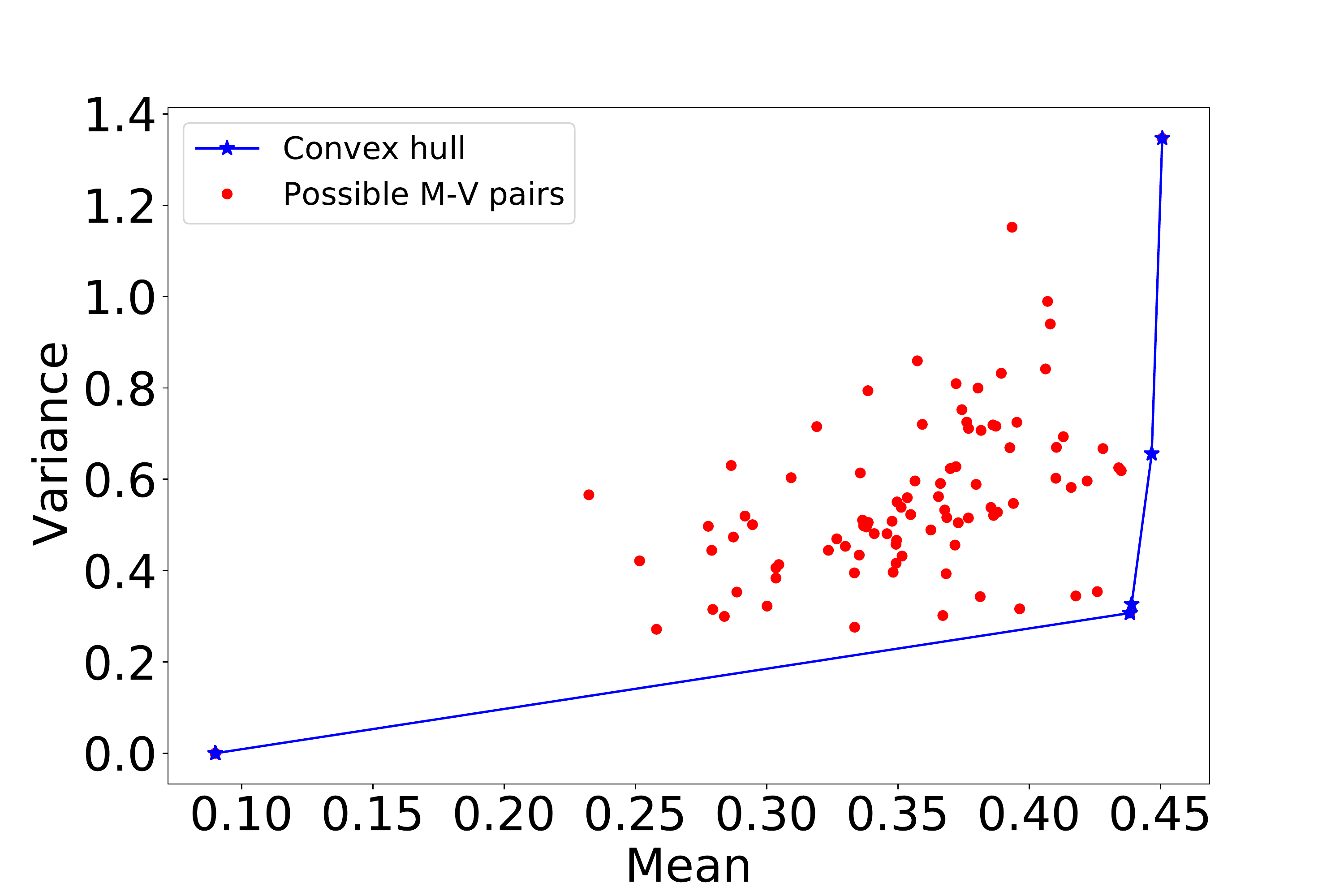}}
	\caption{The convex hull of the risk-averse discounted mean-variance optimization.} \label{fig:PF}
\end{figure}

\subsection{Risk-aversion versus risk-neutral}
Another concern could be the difference between the results of risk-averse and risk-neutral optimizations---``how much profit is sacrificed'' and ``how much risk is eliminated'' are crucial questions.
To quantify the difference, we compare the mean-variance pairs under the optimal policies for the cases with $ \beta = 0 $ and $ \beta = 1 $.
The comparison is shown in Figure~\ref{fig:dist}, with the Gaussian distribution used for presentation only.
The mean and variance for the risk-neutral case is $ (0.4507, 1.3468) $, while its risk-averse counterpart is $ (0.4384, 0.3071) $.
We can see that by sacrificing $ (0.4507 - 0.4384) =0.0123 $ ($ 0.0123/0.4507 \times 100\% \approx 2.73\% $)
in the expected profit, we reduce the variance from 1.3468 to 0.3071 ($ 0.3071/1.3468 \times 100\% \approx 22.80\% $).
It means that in this case, we sacrifice $ 2.73\% $ of the expected profit to eliminate $ (1 - 22.80\%) = 77.20\% $ of the risk, which could be a meaningful risk-averse decision.
\begin{figure}[!h] 
	\centering
	{\includegraphics[width=0.7\columnwidth]{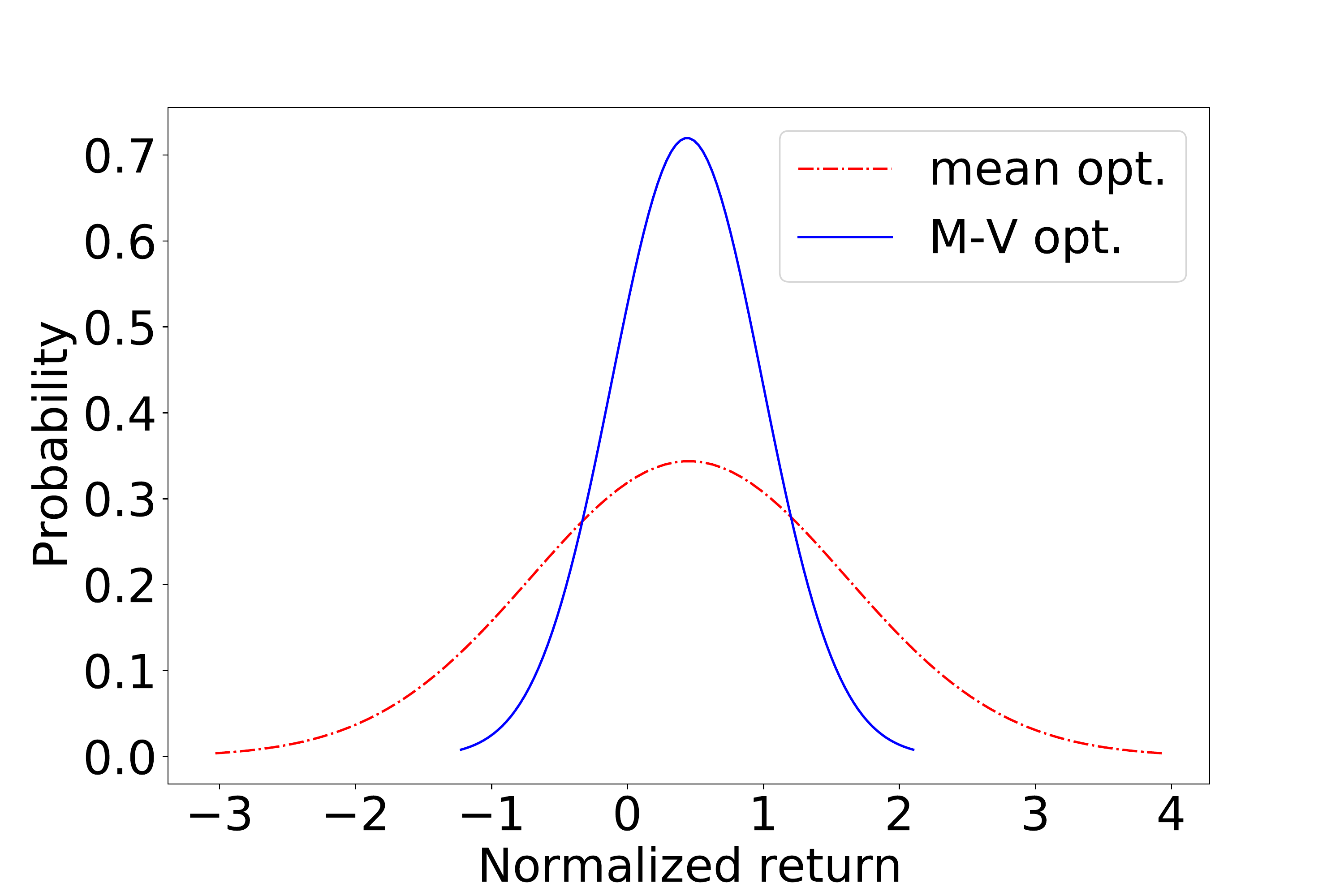}}
	\caption{A comparison between the means and variances with $ \beta = 0 $ and $ \beta = 1 $.} \label{fig:dist}
\end{figure}

\section{Summary and outlook}

In this paper, we study the risk-averse discounted mean-variance optimization, in which the concerned variance refers to the steady-state variance formulated with a discount factor.
The problem formulation unifies the (mean-)variance optimizations in discounted and average MDPs.
Since the reward functions are policy-dependent in the variance-related problems, the MDPs are unorthodox, and the traditional DP methods cannot be applied directly.
We proposed a unified algorithm framework to solve and analyze this problem. 
This framework has a bilevel optimization structure, where the inner problem refers to a standard MDP, and the outer problem concerns a one-dimensional optimization.
For the concerned mean-variance optimization, the outer problem has a closed-form solution, i.e., the discounted mean with respect to a given policy determined by the inner optimization.
This framework unifies a series of algorithms for several variance-related optimizations in discounted and average MDPs, such as the policy gradients~\citep{Prashanth2013NIPS,Bisi2020ijcai}, the policy iterations~\citep{XIA2016,Xia2020,Zhang2021}, and the value iteration~\citep{Gosavi2014}. Furthermore, convergence analyses can be developed with the aid of the Bellman local-optimality equation.
For the risk-averse mean-variance optimization in discounted MDPs, we take value iteration as an example and develop the DMVVI algorithm.
A numerical experiment on a portfolio management problem is given to validate the proposed DMVVI.


Possibilities for future work include studies on the conditions of global convergence, i.e., when an algorithm for a variance-related optimization can converge to a global optimum with probability one.
A first attempt could be a stochastic global convergence achieved with the exploratory mechanism in RL.
The other future work could be model-free RL algorithms, such as Q-learning and SARSA, as online solutions to risk-averse variance-related problems.
We believe that the proposed algorithm framework and one of its consequent algorithms, the DMVVI, provide a theoretic foundation and inspiration for future works.

\end{document}